\documentclass[a4paper,11pt,reqno]{amsart}
\usepackage{amsmath,amsfonts,amsthm,amssymb,enumerate}
\usepackage{graphicx}
\usepackage{comment}

\numberwithin{equation}{section}

\newtheorem{theorem}{Theorem}[section]
\newtheorem{corollary}[theorem]{Corollary}
\newtheorem{proposition}[theorem]{Proposition}

\newtheorem*{fact*}{Fact}
{
\theoremstyle{definition}
\newtheorem{definition}[theorem]{Definition}

\newtheorem{remark}[theorem]{Remark}

}

\title{Closure of knitted surfaces and surface-links}
\author{Inasa Nakamura}
\address{Department of Mathematics, Information Science and  Engineering,  
Saga University,  
1 Honjomachi, Saga, 840-8502, Japan.  
TEL: +81-952-28-8113}
\email{inasa@cc.saga-u.ac.jp}

\author{Jumpei Yasuda}
\address{Department of Mathematics, Graduate School of Science, Osaka University, 1-1, Machikaneyama, Toyonaka, Osaka 560-0043, Japan}
\email{u444951d@ecs.osaka-u.ac.jp}


\begin{document}
\begin{abstract}
A knitted surface is a surface with or without closed components smoothly properly embedded in $D^2 \times B^2$, which is a generalization of a braided surface. 
A knitted surface is called a 2-dimensional knit if its boundary is the closure of a trivial braid.  From a 2-dimensional knit $S$, we obtain a surface-link in $\mathbb{R}^4$ by taking the closure of $S$. We show that any surface-link is ambient isotopic to the closure of some 2-dimensional knit. 
Further, we consider another type of the closure of a knitted surface, called the plat closure. It is known that any trivial surface-knot is ambient isotopic to the plat closure of a knitted surface of degree 2. We show that the plat closure of any knitted surface of degree 2 is a trivial surface-link, and any trivial surface-link is ambient isotopic to the plat closure of a knitted surface of degree $2$. We also show the same result for the closure of 2-dimensional knits of degree 2. 

\end{abstract}
\maketitle

\section{Introduction}\label{sec1}

Let $n$ be a positive integer. 
An {\it $n$-knit} \cite{Birman-Wenzl, Murakami} is a tangle obtained from an $n$-braid in a cylinder $D^2 \times I$ by splicing some crossings. Splicing a crossing results in a \lq\lq hook pair'', and we describe the information of each hook pair 
by a segment called a pairing \cite{NY}. We give an equivalence relation to the set of $n$-knits so that two $n$-knits are equivalent if they are related by an isotopy of $n$-knits. 
Let $p: D^2 \times I \to D^2$ be the projection, and for an $n$-knit $\beta$, we define $\beta_{[t]}=p(\beta \cap (D^2 \times \{t\}))$. 
For two $n$-knits $\beta_0$ and $\beta_1$ in $D^2 \times I$, we define the {\it composition} $\beta_0 \beta_1$ as the $n$-knit given, up to smoothing, by 
\[
(\beta_0 \beta_1)_{[t]}=\begin{cases}\beta_{0[2t]} & (t \in [0,1/2]) \\
\beta_{1[2t-1]} & (t \in [1/2,1]) 
\end{cases}
\]
for $t \in [0,1]$. The set of equivalence classes of $n$-knits forms a monoid $D_n$
by the multiplication induced from the composition of $n$-knits, which is called the \textit{$n$-knit monoid}. 
    The $n$-knit monoid $D_n$ is generated by $3(n-1)$ elements, $\sigma_1, \ldots, \sigma_{n-1}$, $\sigma_1^{-1}, \ldots, \sigma_{n-1}^{-1}$, $\tau_1,\ldots, \tau_{n-1}$, where $\sigma_i$ and $\sigma_i^{-1}$ are the $i$th standard generator and its inverse of the $n$-braid group $B_n$, and $\tau_i$ is the hook pair between the $i$th and $(i+1)$th strings as in Figure \ref{figure1} $(i=1, \ldots, n-1)$. 
  Let $e$ be the unit element of $D_n$, which is represented by the trivial braid. 
  The $n$-knit monoid $D_n$ has the monoid presentation

    \begin{align*}
        \left\langle
        \begin{array}{c|}
            \sigma_1, \ldots, \sigma_{n-1},\\ \sigma_1^{-1}, \ldots, \sigma_{n-1}^{-1},\\ \tau_1,\ldots, \tau_{n-1}
        \end{array}
        \begin{array}{l}
            \sigma_i \sigma_{i}^{-1} = \sigma_i^{-1} \sigma_i = e,~
            \sigma_i \tau_i = \tau_i \sigma_i = \tau_i,\\
            \sigma_i\sigma_j = \sigma_j \sigma_i,~
            \sigma_i \tau_j = \tau_j \sigma_i,~
            \tau_i \tau_j = \tau_j \tau_i \ \ (|i-j|>1), \\
            \sigma_i\sigma_j\sigma_i = \sigma_j\sigma_i\sigma_j,~
            \sigma_i \sigma_j \tau_i = \tau_j \sigma_i \sigma_j \ \ (|i-j|=1)
        \end{array}
        \right\rangle_{\mathrm{monoid}}.
    \end{align*}
    
\begin{figure}[ht]
\includegraphics*[height=3.5cm]{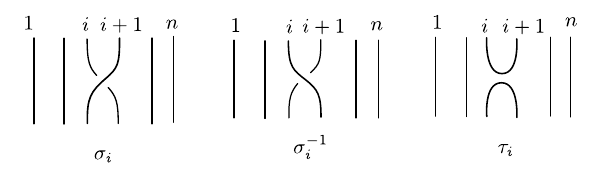}
\caption{Generators $\sigma_i$, $\sigma_i^{-1}$ and $\tau_i$ of the monoid $D_n$ of $n$-knits.
}\label{figure1}
\end{figure}

A knitted surface \cite{NY} in a bidisk $D^2 \times B^2$ is defined 
as an analogue to the notion of an $n$-knit in $D^2 \times I$. 
A knitted surface is an extended notion of a braided surface  \cite{Rudolph}, which is a surface in $D^2 \times B^2$ in the form of a branched covering over $B^2$. A knitted surface of degree $n$ is constructed as the trace of deformations of $n$-knits, as follows. Let $\beta_0 \to \beta_1 \to \beta_2 \to \cdots \to \beta_m$ be a sequence of elements of the free monoid consisting of $\sigma_i, \sigma_i^{-1}$ and $\tau_i$ $(i=1. \ldots, n-1)$ 
such that $\beta_{j-1}$ and $\beta_j$ are related by one of the following for each $j$ $(j=1, \ldots, m)$: 
\begin{enumerate}[(1)]
\item
One of the relations of the $n$-knit monoid $D_n$.  
\item
Replacement of $e$ and $\sigma_i^\epsilon$, or replacement of $e$ and $\tau_i$, for some $i$ and $\epsilon \in \{+1, -1\}$. 
\item
Replacement of $\tau_i$ and $\tau_i \tau_i$ for some $i$. 
\end{enumerate}
We decompose $I=[0,1]$ into $m$ subintervals: $I=I_1 \cup \cdots \cup I_m$, $I_j=[(j-1)/m, j/m]$ $(j=1,\ldots, m)$. 
A knitted surface $S$ of degree $n$ in $D^2 \times I \times I$ is constructed for each $S \cap (D^2 \times I \times I_j)$ as follows. We denote by the same notation $\beta_j$ a geometric $n$-knit given by the presentation $\beta_j$.  
For case (1), let $\{S_{[t]}\}_{t\in I_j}$ be an isotopy of $n$-knits in $D^2\times I$ relating $\beta_{j-1}$ and $\beta_j$ which describes the $n$-knit monoid relation. Then we construct $S \cap (D^2 \times I \times I_j)$ by $\cup_{t \in I_j} S_{[t]} \times \{t\}$. This construction is called the {\it trace} of $\{S_{[t]}\}_{t\in I_j}$. 
The construction is determined by giving $S_{[t]}$, which is called a {\it slice} of $S$ at $t$. 

For case (2), we construct $S \cap (D^2 \times I \times I_j)$ by giving slices as follows: $S_{[t]}=\beta_{j-1}$ $(t \in [(j-1)/m, (2j-1)/2m))$, $S_{[t]}=\beta_j$ $(t \in ((2j-1)/2m, j/m])$ and $S_{[t]}$ $(t=(2j-1)/2m)$ is the union of $\beta_{j-1}$ and a twisted band (respectively, an untwisted band) corresponding to the $\sigma_i^\epsilon$ (respectively $\tau_i$). See Figure \ref{figure-branch}. 
The $n$-knit $\beta_j$ is the result of band surgery along the band attached to $\beta_{j-1}$, and when the band is a twisted band, it corresponds to a branch point of the knitted surface. 

For case (3), let $S_{[t]}=\beta_{j-1}$ $(t \in [(j-1)/m, (2j-1)/2m))$, $S_{[t]}=\beta_j$ $(t \in ((2j-1)/2m, j/m])$ and let $S_{[t]}$ $(t=(2j-1)/2m)$ be the union of $\beta_{j-1}$ and a disk whose boundary is the simple closed curve formed by the $\tau_i \tau_i$. See Figure \ref{figure-tau}. 
The knitted surface is the trace of deformations of $n$-knits where a simple closed curve is born or dies. 
\\

A {\it surface-link} is a closed surface smoothly embedded in $\mathbb{R}^4$. In particular, a surface-link with one component is called a {\it surface-knot}. A braided/knitted surface is called a {\it 2-dimensional braid/knit} if its boundary is the closure of the trivial braid.
From a 2-dimensional braid $S$, we obtain an orientable surface-link in $\mathbb{R}^4$ by taking the closure of $S$,  and Kamada \cite{Kamada94} showed that every orientable surface-link is ambient isotopic to the closure of some 2-dimensional braid.
For a 2-dimensional knit $S$, we can also define the closure of $S$ by a similar way, which provides a surface-link. 
In this paper, we show Claim 9.1 in \cite{NY} as follows. 

\begin{theorem}\label{Theorem: Alexander theorem for surface-links}
    Every surface-link in $\mathbb{R}^4$ is ambient isotopic to the closure of some 2-dimensional knit.
 \end{theorem}

Further, we consider another type of the closure of a knitted surface, called the plat closure. In \cite[Theorem 8.1]{NY}, we showed that any trivial surface-knot is ambient isotopic to the plat closure of some knitted surface of degree 2. In this paper, we show the following.

\begin{theorem}\label{thm20241213}
 The plat closure of any knitted surface of degree 2 is a trivial surface-link, and any trivial surface-link is ambient isotopic to the plat closure of a knitted surface of degree 2. 
 \end{theorem}

Theorem \ref{thm20241213} follows from Corollary \ref{cor5-5} and Theorem~\ref{thm6-7}, where Theorem \ref{thm6-7} is \cite[Claim 8.3]{NY}. 
We show this by using a chart description of a knitted surface of degree 2, which is a generalization of a chart description of a simple braided surface \cite{Kamada96, Kamada92}. As a corollary of Theorem \ref{thm20241213}, we have the following. 

\begin{corollary}\label{thm20241217}
 The closure of any 2-dimensional knit of degree 2 is a trivial surface-link, and any trivial surface-link is ambient isotopic to the closure of some 2-dimensional knit of degree 2. 
 \end{corollary}

The paper is organized as follows.
In Section \ref{sec-knitted-surface}, we review knitted surfaces. 
In Section \ref{sec3}, we review the closure of 2-dimensional knits, and we show Theorem \ref{Theorem: Alexander theorem for surface-links}. 
In Section \ref{sec5}, we review a chart description of knitted surfaces of degree 2. 
In Section \ref{sec8}, we show Theorems \ref{thm6-7} and \ref{thm20241213} and Corollary \ref{thm20241217}. 
We give references \cite{CKS, Kamada02, Kawauchi} for the basics of 1-dimensional and 2-dimensional knot theory. 
In this paper, we assume that surfaces and tangles are smooth.

\section{Knitted surfaces}\label{sec-knitted-surface}
  In this section, we review knitted surfaces \cite{NY}.

\subsection{Motion picture method}\label{sec2}
In order to describe surfaces in 4-space, we use motion pictures, regarding surfaces as the trace of slices of the surfaces by 3-spaces along the 4th dimensional axis.
In this subsection, we review motion picture method.

Let $S$ be a surface properly embedded in $D^2  \times B^2$, and we denote $B^2=I_3 \times I_4$, where $I_3=I_4=[0,1]$.
We modify the saddle points and minimal/maximal points of $S$  into saddle bands and minimal/maximal disks, repsectively; 
see \cite[Section 8.4]{Kamada02}.
Let $p: D^2 \times B^2 \to D^2 \times I_3$ be the projection. 
For each $t \in I_4$, the \textit{slice} of $S$ at $t$, denoted by $S_{[t]}$, is defined as the image by $p$ of the intersection $S \cap (D^2\times I_3 \times \{t\})$: that is, $S_{[t]}=p(S \cap (D^2 \times I_3 \times \{t\}))$. 

When each slice $S_{[t]}$ is either (1) a tangle (or a link) or (2) a union of a tangle and several disks (and bands), we describe $S$ by a finite sequence of diagrams of $S_{[t_0]}, \ldots, S_{[t_m]}$ consisting of unions of (2) and tangles before and after the time of (2). This sequence of diagrams is called a {\it motion picture} of $S$. 

\subsection{Band surgery of a tangle along a band set}
Let $\beta$ be a tangle in $D^2 \times I$.
A {\it band} attached to $\beta$ is a 2-disk $B$ in $D^2 \times I$ such that $B \cap \beta$
consists of a pair of intervals in $\partial B$. The result of {\it band surgery along $B$}, denoted by $h(\beta; B)$, is the tangle 
$\mathrm{Cl}((\beta\cup \partial B)\backslash (\beta \cap B))$. We remark that by a slight modification we assume that $h(\beta; B)$ is also smoothly embedded. 

Let $\mathcal{B}=\{B_1, \ldots, B_m\}$ be a set of mutually disjoint bands attached to $\beta$. Put $\cup \mathcal{B}=B_1 \cup B_2 \cup \cdots \cup B_m$. The result of {\it band surgery along bands of $\mathcal{B}$} is given by 
\[
h(\beta; \mathcal{B}):=\mathrm{Cl}((\beta\cup \partial (\cup \mathcal{B}))\backslash (\beta \cap (\cup \mathcal{B}))). 
\]

\subsection{Knitted surface of degree $n$ presented by a transformation of $n$-knits}
For an $n$-knit $\beta$, we regard $\beta$ as a geometric $n$-knit and we identify $\beta$ with its presentation  $s_1\cdots s_m$ where $s_i\in \{e, \sigma_1^{\pm 1}, \ldots, \sigma_{n-1}^{\pm 1}, \tau_1, \ldots, \tau_{n-1}\}$ for each $i$. 
We call $s_1 \cdots s_m$, $s_i$ ($i=1, \ldots, m$), and $w = s_{n_1}s_{n_1 +1} \cdots s_{n'_1}$ ($1 \leq n_1 \leq n_1' \leq m$) a \textit{word} of $\beta$, a \textit{letter} of $\beta$, and a \textit{subword} of $\beta$, respectively.
        For subwords $w$ and $w'$ of two $n$-knits $\beta$ and $\beta'$, respectively, we call a sequence $w \to w'$ a \textit{transformation} if $\beta'$ is obtained from $\beta$ by changing $w$ to $w'$.
        In this situation, we call a sequence $\beta \to \beta'$ a \textit{transformation} of an $n$-knit.
        For $l$ subwords $w_1, \cdots, w_l$ of $\beta$ and $l$ transformations $w_1 \to w_1'$, \ldots, $w_l \to w_l'$, we say that the transformations are {\it mutually disjoint} if $w_i$ and $w_j$ do not overlap each other for $i \neq j$, that is, when $w_i=s_{n_i} \cdots s_{n_i'}$ and $w_j=s_{n_j} \cdots s_{n_j'}$, the intersection  $[n_i, n_i']\cap [n_j, n_j']$ of intervals is the empty set. 

When we have a transformation of $n$-knits 
$\beta \to \beta'$ which is one of (Case 1)--(Case 3) as below, we construct a knitted surface $S$ of degree $n$ in $D^2 \times B^2$ in such a way as follows; since the construction is unique up to equivalence, we call $S$ the {\it knitted surface presented by $\beta \to \beta'$}. In the following (Case 2) and (Case 3), $\beta \leftrightarrow \beta'$ denotes the sequences $\beta \to \beta'$ and $\beta' \to \beta$. 

\noindent
{\bf (Case 1)} A transformation $\beta \to \beta'$, where $\beta$ and $\beta'$ is related by an isotopy of $n$-knits $\{\beta_t\}_{t \in [0,1]}$: $\beta_0=\beta$ and $\beta'=\beta_1$. 

In this case, we define $S$  by 
\[
S_{[t]}=\beta_t
\]
for any $t \in [0,1]$. 

\noindent
{\bf (Case 2)} A transformation $e \leftrightarrow \sigma_i^\epsilon$ or $e \leftrightarrow \tau_i$ $(\epsilon \in \{+1, -1\}$). This is a transformation $\beta \to \beta'$ such that 
$\beta'$ is obtained from $\beta$ by band surgery along a band as in Figure \ref{figure-branch}. 

\begin{figure}[ht]
\includegraphics*[height=3cm]{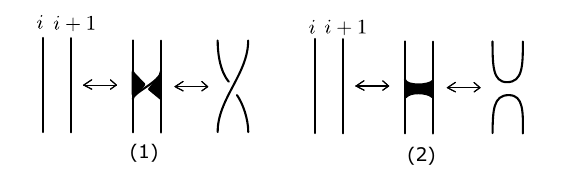}
\caption{Motion pictures of the knitted surfaces presented by (1) $e \leftrightarrow \sigma_i$, and (2) $e \leftrightarrow \tau_i$, where we omit the $j$th strings for $j \neq i, i+1$. This figure is \cite[Figure 2]{NY}.
}
\label{figure-branch}
\end{figure}

In this case, we define $S$ by the motion picture as given in Figure \ref{figure-branch}; we define $S$ by slices as follows: for $e \to s_i$ $(s_i \in \{\sigma_i^\epsilon, \tau_i\})$, we denote by $B$ a band attached to $e$ so that the result of the band surgery along $B$ is $s_i$. Then we define $S$ by 
\[
S_{[t]}=\begin{cases}
e & (0 \leq t <\frac{1}{2})\\
e \cup B & (t=\frac{1}{2})\\
s_i & (\frac{1}{2}<t\leq 1),
\end{cases}
\]
and for $s_i \to e$, the presented surface $S'$ is given by the mirror image of $S$ with respect to $D^2 \times I_3 \times \{1/2\}$, that is, the surface $S'$ given by the slice $S'_{[t]} =S_{[1-t]}$ for each $t \in [0,1]$. 

\noindent
{\bf (Case 3)} A transformation $\tau_i \leftrightarrow \tau_i\tau_i$. 
In this case, we define $S$ by the motion picture as given in Figure \ref{figure-tau}; we define $S$ by slices as follows: we define $S$ for $\tau_i \to \tau_i\tau_i$; the other case is given by the same method as in (Case 2). We denote by $D$ a disk whose boundary is the simple closed curve in $\tau_i \tau_i$. Further, we assume that $\tau_i=\tau_i \tau_i \backslash \partial D$. Then we define $S$ by 
\[
S_{[t]}=\begin{cases}
\tau_i & (0 \leq t <\frac{1}{2})\\
\tau_i \cup D & (t=\frac{1}{2})\\
\tau_i \tau_i & (\frac{1}{2}<t\leq 1). 
\end{cases}
\]

\begin{figure}[ht]
\includegraphics*[height=3cm]{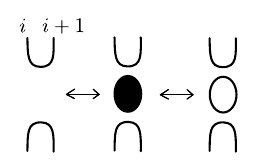}
\caption{Motion picture of the knitted surface presented by $\tau_i \leftrightarrow \tau_i\tau_i$, where we omit the $j$th strings for $j \neq i, i+1$. This figure is similar to \cite[Figure 3]{NY}. 
}\label{figure-tau}
\end{figure}

For a transformation $\beta \to \beta'$ such that $\beta'$ is obtained from $\beta$ by a finite number of mutually disjoint transformations of (Case 1)--(Case 3), the knitted surface presented by $\beta \to \beta'$ is defined as the surface determined from those of (Case 1)--(Case 3). 
We remark that the construction is called the \lq\lq vertical product'' \cite[Section 15.1]{Kamada02}.

\subsection{Knitted surfaces}\label{subsec-knitted-surface}

Knitted surfaces are defined by using \lq\lq pairings'' and \lq\lq knit structure'' \cite{NY}. Two knitted surfaces are said to be equivalent if they are related by an \lq\lq isotopy of knitted surfaces''. Every knitted surface is equivalent to a knitted surface in a normal form. 
In this paper, we give a knitted surface as a knitted surface of degree $n$ derived from a normal form, by using deformations of $n$-knits.

\begin{definition}\label{Definition: knitted surface in normal form}
Let $S$ be a surface properly embedded in $D^2 \times B^2$. 
Then, $S$ is called a {\it knitted surface} of degree $n$ if there exists a partition $0=t_0<t_1< \ldots <t_{m}=1$ for some $m$ satisfying the following conditions, where $\beta_k=S_{[t_k]}$ $(k=0,\ldots, m)$. 
    \begin{enumerate}
        \item Each slice $\beta_k$ is an $n$-knit $(k=0, \ldots, m)$. 
        \item The surface $S \cap (D^2 \times I_3 \times [t_{k-1}, t_k])$ is the knitted surface presented by $\beta_{k-1} \to \beta_k$ ($k=1, \dots, m$) such that the knit $\beta_{k}$ is obtained from $\beta_{k-1}$ by a finite number of mutually disjoint transformations associated with the relations of the $n$-knit monoid $D_n$ as follows $(k=1, \ldots, m)$, where $\beta \leftrightarrow \beta'$ denotes the sequences $\beta \to \beta'$ and $\beta' \to \beta$:
\begin{itemize}
\item \ $e \leftrightarrow \sigma_i^\epsilon$,  
\item \  $e \leftrightarrow \tau_i$,   
\item $\tau_i \tau_i \leftrightarrow \tau_i, $
\item $\tau_i \leftrightarrow \sigma_i^\epsilon \tau_i,$
\ $\tau_i \leftrightarrow \tau_i \sigma_i^\epsilon,$
\item $\sigma_i^\epsilon \sigma_i^{-\epsilon} \leftrightarrow e$, 
\item 
$\sigma_i^\epsilon \sigma_j^\delta \leftrightarrow \sigma_j^\delta \sigma_i^\epsilon \quad  (|i-j|>1),  $
\item
$\sigma_i^\epsilon \tau_j \leftrightarrow \tau_j \sigma_i^\epsilon \quad  (|i-j|>1)$,  \item $\tau_i \tau_j \leftrightarrow \tau_j \tau_i \quad  (|i-j|>1)$,  \
\item $\sigma_i^\epsilon \sigma_j^\epsilon \sigma_i^\epsilon \leftrightarrow \sigma_j^\epsilon \sigma_i^\epsilon \sigma_j^\epsilon,\ \sigma_i^\epsilon \sigma_j^\epsilon \sigma_i^{-\epsilon} \leftrightarrow \sigma_j^{-\epsilon} \sigma_i^{\epsilon} \sigma_j^\epsilon \quad  (|i-j|=1),$ 
\item $\sigma_i^\epsilon \sigma_j^\epsilon \tau_i \leftrightarrow \tau_j \sigma_i^\epsilon \sigma_j^\epsilon  \quad  (|i-j|=1)$,  
\end{itemize}
where $\epsilon,\delta \in \{+1, -1\}$, $i,j \in \{1,\ldots, n-1\}$.
\end{enumerate}

We denote by $Q_n$ the starting point set of $n$-knits, that is a fixed set of $n$ interior points of $D^2$ such that $\beta \cap (D^2 \times \partial I_3)=Q_n \times \partial I_3$ for any $n$-knit $\beta$. 
We call the knitted surface $Q_n \times B^2$ the {\it trivial} knitted surface.

A knitted surface of degree $n$ is called a {\it braided surface} of degree $n$ if the transformations $\beta_0 \to \cdots \to \beta_m$ are restricted to the relations of the $n$-braid group $B_n$; that is, the relations do not contain $\tau_i$ $(i=1, \ldots, n-1)$.

A knitted surface $S$ of degree $n$ is called a \textit{2-dimensional knit} if $\partial S=Q_n \times \partial B^2$. 
A 2-dimensional knit is called a \textit{2-dimensional braid} if it is a braided surface. 

\end{definition}

\section{The closure of 2-dimensional knits}\label{sec3}

In this section, we show Theorem \ref{Theorem: Alexander theorem for surface-links}. 

\subsection{Closure of knits and 2-dimensional knits}
Let $B^m_1$ and $B_2^m$ be two copies of an $m$-disk $B^m$ $(m=1,2)$. Let $S$ be an $n$-knit in $D^2 \times B_1^1$ or a 2-dimensional knit in $D^2 \times B_1^2$, and let $S^m_0 = Q_n \times B^m_2 \subset D^2 \times B_2^m$ be the trivial knits/2-dimensional knits. Then, $\mathrm{cl}(S):= S \cup S_0^m$ is a closed 1-manifold or a  closed surface in $D^2 \times S^m$, where $S^m = B_1^m \cup_\partial B_2^m$. 

Let $F_0^m$ be a circle (respectively a 2-sphere) standardly embedded in $\mathbb{R}^{m+2}$ when $m=1$ (respectively $m=2$), and let $N_m$ be a regular neighborhood of $F_0^m$. 
We identify $N_m$ with $D^2 \times S^m$ $(m=1,2)$.
The \textit{closure} of $S$ is a link or a surface-link in $\mathbb{R}^{m+2}$ which is given by $\mathrm{cl}(S) \subset N_m \subset \mathbb{R}^{m+2}$.
 
As an analogue of Alexander's theorem for oriented links, it is known that: 

\begin{theorem}[\cite{Kamada94-2}]\label{Theorem: Alexander theorem for orientable surface-links}
    Every orientable surface-link is ambient isotopic to the closure of some 2-dimensional braid.
\end{theorem}

\subsection{The plat closure of braids}
For a set $w=\{w^1, \ldots, w^l\}$, we denote by $\cup w$ the union $w^1 \cup \cdots \cup w^l$. 
A \textit{wicket} is a semi-circle properly embedded in $D^2 \times I$ which meets $D^2 \times \{0\}$ orthogonally.
A \textit{configuration of $n$ wickets} is a set of mutually disjoint $n$ wickets.
$\mathcal{W}_n$ denotes the space consisting of configurations of $n$ wickets.
Let $Q_{2n}$ be a set of $2n$ interior points $\{q_1 ,\dots, q_{2n}\}$ of $D^2$ such that $q_1, \dots, q_{2n}$ lie in $\{1/2\} \times I \subset I \times I=D^2$ in this order. Let $w_0$ denotes the configuration of $n$ wickets such that $\partial w_0 = Q_{2n} \times \{0\}$ and $q_{2i-1}$ and $q_{2i}$ are boundaries of the same component of $w_0$ for each $i \in \{1, \dots, n\}$. 
We denote by $w_0^*$ the mirror image of $w_0$ with respect to $D^2 \times \{1/2\}$; remark that $\cup w_0$ and $\cup w_0^*$ are disjoint.  
The union $ (\cup w_0) \cup (\cup w_0^*)$ represents the $2n$-knit $\tau := \tau_1 \tau_3 \dots \tau_{2n-1}$. 

For a $2n$-braid $\beta$, the \textit{plat closure} of $\beta$, denoted by $\widetilde{\beta}$, is a link in $D^2\times [-1,2] \subset \mathbb{R}^3$ obtained by taking the union of $w_0 \subset D^2\times [-1,0]$, $\beta \subset D^2\times [0,1]$, and $w_0^* \subset D^2\times [1,2]$. 
See Figure~\ref{20240707-1}.
The plat closure $\widetilde{\beta}$ can be regarded as the closure of the knit $\tau\beta$, that is, $\widetilde{\beta}=\mathrm{cl}(\tau \beta)$.

\begin{figure}[h]
    \centering
    \includegraphics[width=0.7\hsize]{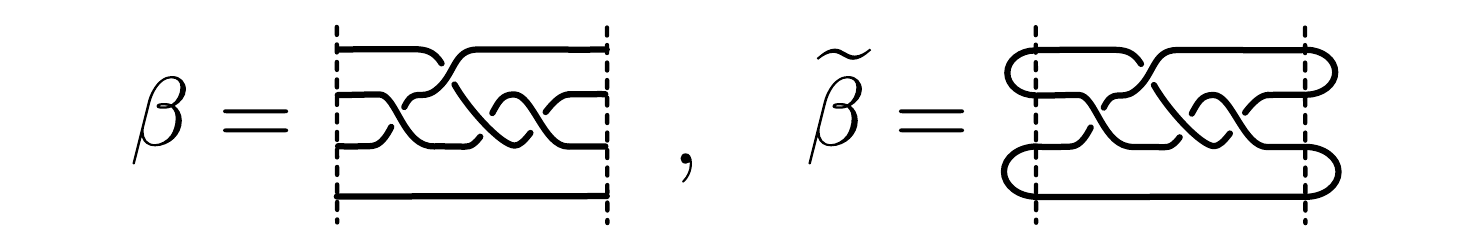}
    \caption{The plat closure of a braid. }
    \label{20240707-1}
\end{figure}

\subsection{The plat closure of braided surfaces}\label{sec:pcl-braided-surface}
We recall the definition of the plat closure of a braided surface \cite{Yasuda21}. To take the plat closure, a braided surface is required to be \lq\lq adequate''. 

For a loop $f: (I, \partial I) \to (\mathcal{W}_n, w_0)$, we define the $2n$-braid $\beta_f$ by $\beta_f = \bigcup_{t \in I} \partial f(t) \times \{t\}  \subset D^2\times \{0\} \times I$.
A braid $\beta$ is called \textit{adequate} if there exists a loop $f: (I, \partial I) \to (\mathcal{W}_n, w_0)$ such that $\beta = \beta_f$. 
The set of adequate $2n$-braids forms a group called the {\it Hilden subgroup} of the $2n$-braid group $B_{2n}$, denoted by $K_{2n}$, and $K_{2n}$ is generated by $\sigma_{2i-1}$, $\sigma_{2j}\sigma_{2j-1}\sigma_{2j+1}\sigma_{2j}$, and $\sigma_{2j}\sigma_{2j-1}\sigma^{-1}_{2j+1}\sigma^{-1}_{2j}$ for $i \in \{1, \dots, n\}$ and $j \in \{ 1, \dots, n-1\}$ \cite{Brendle-Hatcher2008} (see Figure \ref{20240707-2} for the case $n=2$).
 
\begin{figure}[h]
    \centering
    \includegraphics[width = 0.6\hsize]{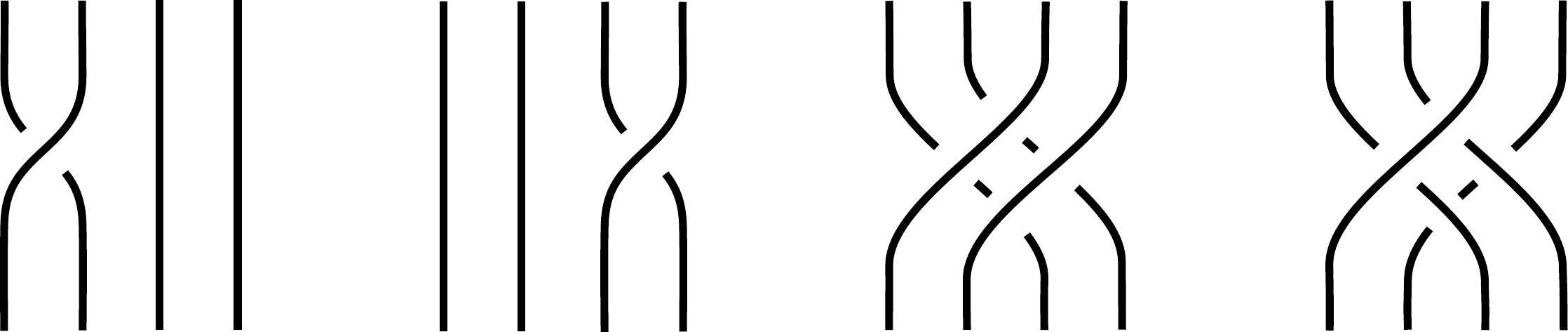}
    \caption{Adequate braids which generate the Hilden subgroup $K_4$: from the left, $\sigma_{1}$, $\sigma_3$, $\sigma_{2}\sigma_{1}\sigma_{3}\sigma_{2}$, and $\sigma_{2}\sigma_{1}\sigma^{-1}_{3}\sigma^{-1}_{2}$. This figure is  \cite[Figure 16]{NY}. 
  }
    \label{20240707-2}
\end{figure}

Let $\beta$ be an adequate $2n$-braid in $D^2 \times I$. 
There exists a unique loop $f: (I, \partial I) \to (\mathcal{W}_n, w_0)$ satisfying $\beta_f = \beta$. Let $S^1 = [0,1]/\sim$ with $0 \sim 1$. Let $N=I \times S^1$, an annulus. 
We define the surface $A_\beta$ in  $D^2 \times N$ by 
\begin{equation}\label{eq1218}
A_\beta = \bigcup_{t \in [0,1]} (f(t)\times \{t\}) \subset (D^2 \times I) \times S^1=D^2 \times N. 
\end{equation}

A braided surface $S$ is called \textit{adequate} if $\partial S$ is the closure of an adequate braid. We remark that an adequate braided surface is of degree $2n$ for some $n$, and every 2-dimensional braid of degree $2n$ is adequate. 

For an adequate braided surface $S$ in $D^2 \times B^2$, we denote by $\beta$ the adequate braid whose closure is $\partial S$. We take the surface $A_\beta$ given by (\ref{eq1218}). 
Recall that $B^2=[0,1] \times [0,1]$. We identify $N$ with $[-1,2] \times [-1,2] \backslash B^2$, and we identify $\partial B^2$ with $\{0\} \times S^1\subset I \times S^1=N$, and we embed $D^2 \times N \subset \mathbb{R}^4$ by $D^2 \subset  \mathbb{R}^2$ and $N \subset \mathbb{R}^2$.   
We remark that $\partial S = \partial A_\beta = S \cap A_\beta$.
Then the \textit{plat closure} of $S$ is a surface-link in $\mathbb{R}^4$ defined by the union of $S$ and $A_\beta$, denoted by $\widetilde{S}$.

\begin{theorem}[\cite{Yasuda21}]\label{TheoremY}
    Every surface-link is ambient isotopic to the plat closure of some adequate braided surface.
  \end{theorem}

\subsection{Proof of Theorem~\ref{Theorem: Alexander theorem for surface-links}}

Let $F$ be a surface-link. By Theorem \ref{TheoremY}, we take  an adequate braided surface $S$ of degree $2n$ such that the plat closure $\widetilde{S}$ is ambient isotopic to $F$.

\noindent
{\bf (Case 1)}  
First, we consider the case that $S$ is a 2-dimensional braid, that is, $\widetilde{S}$ has the motion picture as in Figure~\ref{20240629-1}. Put $\tau=\tau_1 \tau_3 \cdots \tau_{2n-1}$. 
\begin{figure}[h]
    \centering
    \includegraphics{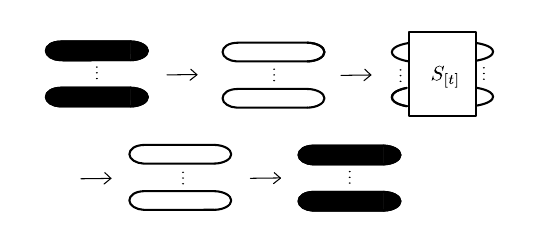}
    \caption{Motion picture of the plat closure of a 2-dimensional braid $S$. }
    \label{20240629-1}
\end{figure}
We take a surface-link $F_1$ given by 
\[
(F_1)_{[t]}=\begin{cases}
\mathcal{D} & (t=-1)\\
\mathrm{cl}(\tau) & (-1 < t<0)\\
\mathrm{cl}(\tau S_{[t]}) & (0 \leq t\leq 1)\\
\mathrm{cl}(\tau) & (1<t<2)\\
\mathcal{D} & (t=2),
\end{cases}
\]
where $\mathcal{D}$ is a disjoint union of $n$ copies of a disk whose boundaries are $\mathrm{cl}(\tau)$; see Figure \ref{20240629-2}. 
The surface $F_1$ is ambient isotopic to 
the plat closure $\widetilde{S}$. 
\begin{figure}[h]
    \centering
    \includegraphics{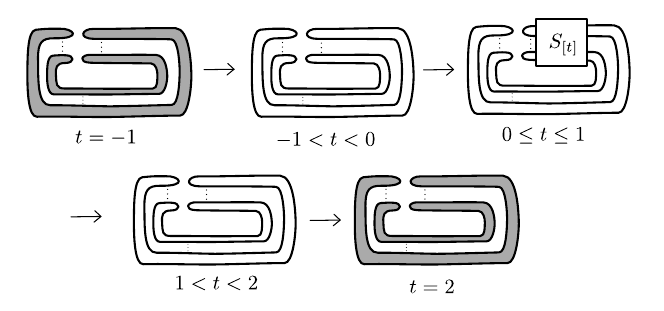}
    \caption{Motion picture of $F_1$. }
    \label{20240629-2}
\end{figure}

Let $\Delta^-$ be the surface properly embedded in $\mathbb{R}^3 \times (-\infty, 0]$ given by
\[
(\Delta^-)_{[t]}=\begin{cases}
\mathcal{D'} & (t=-2)\\
\mathrm{cl}(e) & (-2 < t<-1)\\
\mathrm{cl}(e) \cup (\cup\mathcal{B}) & (t=-1)\\
\mathrm{cl}(\tau) & (-1<t\leq 0),
\end{cases}
\]
where $\mathcal{D}'$ is a disjoint union of $2n$ disks whose boundaries are $\mathrm{cl}(e)$ and $\mathcal{B}$ is a band set attached to $e$ such that $h(e; \mathcal{B})=\tau$, and $\cup \mathcal{B}$ is the union of bands of $\mathcal{B}$; see Figure \ref{20240629-3}. The surface $\Delta^-$ consists of $n$ copies of a disk. 
We denote by $F^-_1$ the surface $F_1 \cap (\mathbb{R}^3 \times (-\infty, 0])$. 
Then both $\Delta^-$ and $F_1^-$ consist of $n$ copies of a disk with $\Delta^-\cap (\mathbb{R}^3 \times \{0\})=\partial \Delta^-=\partial F_1^-=F_1^- \cap (\mathbb{R}^3 \times \{0\}) \subset \mathbb{R}^3\times \{0\}$. These sets of disks are called {\it trivial disk systems}, and it is known \cite{KSS} that they are ambient isotopic in $\mathbb{R}^3 \times (-\infty, 0]$ rel $\mathbb{R}^3 \times \{0\}$.

\begin{figure}[h]
    \centering
    \includegraphics{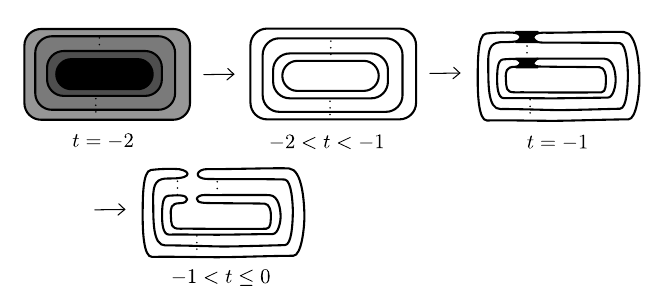}
    \caption{Motion picture of $\Delta^-$. }
    \label{20240629-3}
\end{figure}

Similarly, we take the surface $\Delta^+$ properly embedded in $\mathbb{R}^3 \times [1, \infty)$ given by the mirror image of $\Delta^-$ with respect to $\mathbb{R}^3 \times \{1/2\}$, and we denote by $F_1^+$ the surface $F_1 \cap (\mathbb{R}^3 \times [1, \infty))$. 
Then $\Delta^+$ and $F_1^+$ are ambient isotopic in $\mathbb{R}^3 \times [1, \infty)$ rel $\mathbb{R}^3 \times \{1\}$.
Let $F_2$ be the surface obtained from $F_1$ by replacing $F_1^+$ and $F_1^-$ with $\Delta^+$ and $\Delta^-$, respectively; that is, 
\[
F_2=\Delta^- \cup (F_1 \cap (\mathbb{R}^3 \times [0,1])) \cup \Delta^+. \]
By construction, $F_2$ is ambient isotopic to $F_1$.

Let $D^2 \times I$ be the cylinder associated with the knits in slices. 
Put $S_1=F_2 \cap (D^2 \times I \times [-3/2, 5/2])$.
Then $S_1$ is a 2-dimensional knit given by 
\[
(S_1)_{[t]}=\begin{cases}
e & (-3/2 \leq t \leq -1)\\
e \cup (\cup \mathcal{B}) & (t=-1)\\
\tau & (-1< t <0)\\
\tau S_{[t]} & (0 \leq t \leq 1)\\
\tau & (1 < t< 2)\\

\tau \cup (\cup \mathcal{B}) & (t=2)\\
e & (2 <t\leq 5/2). 
\end{cases}
\]
Further, we see that the closure of $S_1$ is $F_2$. Thus, the closure of the 2-dimensional knit $S_1$ is ambient isotopic to the given surface-link $F$. 

\noindent
{\bf (Case 2)} 
Next, we consider the case when $F$ is ambient isotopic to the plat closure $\tilde{S}$ of an adequate braided surface $S\subset D^2 \times B^2$ of degree $2n$.
 By considering a self-homeomorphism $f$ of $B^2$ and the induced self-homeomorphism of $D^2 \times B^2$ given by $\mathrm{id} \times f$, we assume that $\partial S \cap (D^2 \times I \times \{1\})$ has a presentation of an adequate braid  whose closure is $\partial S$. We denote the adequate braid $\partial S \cap (D^2 \times I \times \{1\})$ by $\gamma$. 
 By the same isotopy used to obtain $F_1$ from $\tilde{S}$ in (Case 1), we have a surface-link $F_1$. 
 Let $S'$ be a knitted surface in $D^2 \times B^2$ associated with the adequate braid $\gamma$ given in \cite[Section 6.4.2]{NY}. Then, the surface $S'$ satisfies  
 \[
S'_{[t]}=\begin{cases}
 \tau \gamma& (t=0)\\
\tau & (t=1)\\
\end{cases}. 
\]
Thus, let $S_1$ be the knitted surface given by 
 \[
 (S_1)_{[t]}=\begin{cases}
 \tau S_{[2t]} & (0 \leq t \leq 1/2)\\
 S'_{[2t-1]} & (1/2 \leq t \leq 1). 
 \end{cases}
 \]
Let $F_2$ be the 
surface given by 
\[
 (F_2)_{[t]}=\begin{cases}
 \mathcal{D} & (t=-1)\\
 \mathrm{cl}(\tau) & (-1<t<0)\\
 \mathrm{cl}((S_1)_{[t]}) & (0 \leq t\leq 1)\\
 \mathrm{cl}(\tau) & (1 <t<2))\\
 \mathcal{D} & (t=2),  
 \end{cases}
 \]
 where $\mathcal{D}$ is a disjoint union of $n$ copies of a disk whose boundaries are $\mathrm{cl}(\tau)$.

By construction of $S'$, $F_2 \cap (\mathbb{R}^3 \times [1/2, \infty))$ is a trivial disk system (see \cite[Proposition 6.6]{NY}) and $F_2$ is ambient isotopic to the surface-link $F_1$. 

Then, by the same process with (Case 1), we obtain a new surface-link $F_3$ from $F_2$ by replacing $F_2^+ = F_2 \cap (\mathbb{R}^3 \times [1, \infty))$ and $F_2^- = F_2 \cap (\mathbb{R}^3 \times (-\infty, 0])$ with $\Delta^+$ and $\Delta^-$, respectively. The surface-link $F_3$ is ambient isotopic to $F_2$. 
Then, $S_2:=F_3 \cap (D^2 \times I \times [-3/2, 5/2])$ is a 2-dimensional knit and the closure of $S_2$ is the surface-link $F_3$. 
Hence, the closure of the 2-dimensional knit $S_2$ is ambient isotopic to the given surface-link $F$.
\qed

\section{Chart description}\label{sec5}

A knitted surface of degree $n$ is presented by a finite graph on $B^2$ called a \lq\lq chart'' of degree $n$ \cite{NY}, which is a generalization of a chart for braided surfaces \cite{Kamada96, Kamada94, Kamada92}. 
Here, we give the definition of chart of degree 2. See \cite{NY} for the definition of degree $n$. 

\begin{definition}\label{def5-5}
Let $\Gamma$ be a finite graph in a 2-disk $B^2$.
Then, $\Gamma$ is a {\it knit chart} of degree $2$, or simply a {\it 2-chart}, if it satisfies the following conditions.

\begin{enumerate}
\item
The intersection $\Gamma \cap \partial B^2$ consists of a finite number of endpoints of edges of $\Gamma$ meeting $\partial B^2$ orthogonally. Though they are vertices of degree one, we call elements of $\Gamma \cap \partial B^2$ {\it boundary points}, and we call only a vertex in $\mathrm{Int}(B^2)$ {\it a vertex of $\Gamma$}.
 
\item
Each edge is equipped with the label $1$ and moreover, each edge is either oriented or unoriented. We call an oriented/unoriented edge a {\it $\sigma$-/$\tau$-edge}, respectively. 

\item
Each vertex is of degree 1 or 3 as depicted in Figure  \ref{BMWchart}.

A vertex of degree 1 connected with a $\sigma$-edge  (respectively $\tau$-edge)  is depicted by a small black (respectively red) disk, called a {\it black $\sigma$-vertex} (respectively {\it black $\tau$-vertex}). A vertex of degree $3$ is called a {\it mixed trivalent vertex} (respectively a {\it trivalent $\tau$-vertex}) if it is connected with a $\sigma$-edge and a pair of $\tau$-edges (respectively three $\tau$-edges). 

\end{enumerate}

\end{definition}

\begin{figure}[ht]
\includegraphics*[height=4cm]{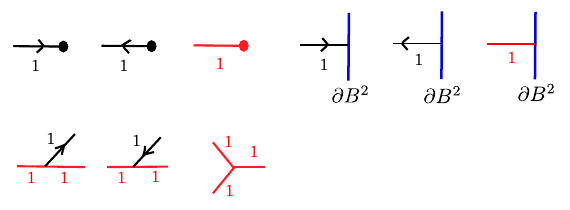}
\caption{Vertices of a 2-chart.}%
\label{BMWchart}
\end{figure}

A knitted surface has a chart description. 
From a knitted surface of degree $n$, we obtain a chart of degree $n$, and 
from a given chart $\Gamma$, 
we can construct a knitted surface whose chart is $\Gamma$ (\cite{NY}). 

In the rest of this paper, we focus on knitted surfaces of degree $2$ and their $2$-charts. Thus we only
 review the knitted surface $S(\Gamma)$ of degree 2 constructed from a $2$-chart $\Gamma$ as follows. 
We decompose $B^2$ into a union of a finite number of 2-disks such that for each disk $E$, by modification, $\Gamma \cap E$ is either an empty graph (called an {\it empty chart}) or one of the $2$-charts as in Figures \ref{BMWchart} and \ref{chart-edge2}, where the dotted square in Figure \ref{chart-edge2} denotes $\partial E$. 
We identify $E$ with $I_3 \times I_4$, where the direction of $I_3$ is the vertical direction from up to down and the direction of $I_4$ is the horizontal direction from left to right in the figures. 
When $\Gamma \cap E$ is an empty chart, we define $S(\Gamma \cap E)$ by the trivial knitted surface $Q_2 \times E$. 
When $\Gamma \cap E$ is one of the $2$-charts as in Figure  \ref{chart-edge2}, we define $S(\Gamma \cap E)$ by the knitted surface $\sigma_1 \times I_4$ for Figure \ref{chart-edge2}(1), $\sigma_1^{-1} \times I_4$ for Figure \ref{chart-edge2}(2), and $\tau_1 \times I_4$ for Figure \ref{chart-edge2}(3). 
When $\Gamma \cap E$ is one of the $2$-charts as in Figure  \ref{BMWchart}, we define $S(\Gamma \cap E)$ by the knitted surface presented by $s \to e$ for the $2$-chart containing one black $\sigma$/$\tau$-vertex for appropriate $s \in \{\sigma_1, \sigma_1^{-1}, \tau_1\}$, and the knitted surface presented by $\tau_1 \to \sigma_1 \tau_1$, $\tau_1 \to \sigma_1^{-1}\tau_1$, or $\tau_1\tau_1 \to \tau_1$ for the $2$-chart containing a trivalent vertex. 
We construct the knitted surface presented by $\Gamma$ as the union of the knitted surfaces presented by $\Gamma \cap E$. 

\begin{figure}[ht]
\includegraphics*[height=3cm]{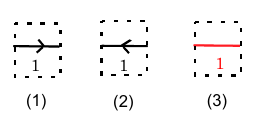}
\caption{The $2$-chart $\Gamma \cap E$ consisting of one edge.}%
\label{chart-edge2}
\end{figure}

\subsection{Chart moves}\label{sec7-1}
In this section, we give the definition of chart moves for $2$-charts, that are local moves of charts. This includes C-moves of charts of braided surfaces \cite{Kamada02, Kamada96, Kamada92}.

\begin{definition}\label{def6-2}
Let $\Gamma$ and $\Gamma'$ be $2$-charts in $B^2$. 
Then, a {\it knit chart move}, or a {\it C-move}, is a local move in a 2-disk $E \subset B^2$ such that $\Gamma$ (respectively $\Gamma'$) is changed to $\Gamma'$ (respectively $\Gamma$), satisfying the following.

\begin{enumerate}
\item
The boundary $\partial E$ does not contain any vertices of 
$\Gamma$ and $\Gamma'$ except for boundary points, and $\Gamma \cap \partial E$ (respectively $\Gamma' \cap \partial E$) consists of transverse intersection points of edges of $\Gamma$ (respectively $\Gamma'$) and $\partial E$.

\item
The charts are identical in the complement of $E$:
$\Gamma \cap (B^2\backslash E)=\Gamma' \cap (B^2\backslash E)$.

\item
The charts in $E$, $\Gamma \cap E$ and $\Gamma' \cap E$, 
are as in Figure \ref{20241211-1a} and the orientation-reversed figures of Figure \ref{20241211-1a}. 
 
\end{enumerate}

We say that $\Gamma$ and $\Gamma'$ are {\it C-move equivalent} if they are related by a finite sequence of C-moves and ambient isotopies of $B^2$. 
 
\end{definition}

\begin{figure}[ht]
\includegraphics*[height=8cm]{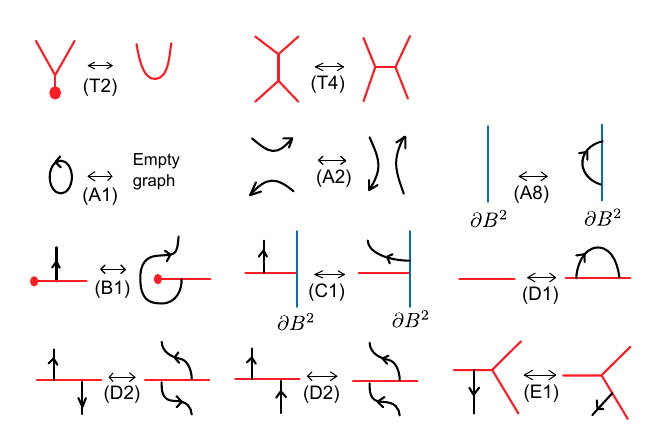}
\caption{C-moves for 2-charts, where we omit labels, which are all $1$. C-moves include the orientation-reversed version. }
\label{20241211-1a}
\end{figure}

Let $\Gamma$ and $\Gamma'$ be $2$-charts and let $S$ and $S'$ be the knitted surfaces presented by $\Gamma$ and $\Gamma'$, respectively. 
Theorem 7.5 in \cite{NY} implies that $S$ and $S'$ are ambient isotopic if $\Gamma$ and $\Gamma'$ are $C$-move equivalent. Thus, the plat closures of $S$ and $S'$ are ambient isotopic if $\Gamma$ and $\Gamma'$ are $C$-move equivalent.

\section{Plat closure of knitted surfaces of degree $2$}\label{sec8}

In this section, we consider knitted surfaces of degree $2$. 
We give a normal form of a $2$-chart (Definition  \ref{normal-form-chart} and Proposition \ref{thm6-3}). 
We show that any trivial surface-link is ambient isotopic to the plat closure of knitted surfaces of degree 2 (Corollary \ref{cor5-5}).
We show that the plat closure of any knitted surface of degree 2 is a trivial surface-link (Theorem \ref{thm6-7}), and we show Theorem \ref{thm20241213} and Corollary \ref{thm20241217}. 
\\

\subsection{Knitted surfaces of degree $2$}

 A $\sigma$-/$\tau$-edge of a chart is called a {\it free $\sigma$-/$\tau$-edge} if its endpoints are black $\sigma$-/$\tau$-vertices. A $\sigma$-edge of a chart
 is called a {\it half $\sigma$-edge} if its endpoints consists of a black $\sigma$-vertex and either a boundary point or a mixed
 trivalent vertex. A $\tau$-edge of a chart is called a {\it half $\tau$-edge} if its endpoints consists of a black $\tau$-vertex and either a boundary point or a trivalent $\tau$-vertex. 
  
For a chart $\Gamma$ of degree 2, the {\it $\tau$-chart} of $\Gamma$, denoted by $\Gamma_\tau$, is defined as the subgraph of $\Gamma$ consisting of $\tau$-edges and the connecting vertices. The $\tau$-chart $\Gamma_\tau$ may contain $2$-valent vertices which are obtained from mixed
 trivalent vertices. Then, we merge the two edges connected by each 2-valent vertex after removing the $2$-valent
 vertex. Thus, $\Gamma_\tau$ consists of intervals, circles, and graphs consisting of $1$-valent vertices and $3$-valent
 vertices. We call a connected component of $\Gamma_\tau$ {\it a $\tau$-component}. When $\Gamma_\tau$ is contained in $B^2$, a {\it region} of $\Gamma_\tau$ is the closure of a connected component of $B^2\backslash \Gamma_\tau$.

\begin{definition}\label{normal-form-chart}
Let $\Gamma$ be a $2$-chart, let $R_1, \ldots,R_m$ be  regions of the $\tau$-chart $\Gamma_\tau$, and put $\Gamma^j := \Gamma \cap R_j$, the subchart of $\Gamma$ $(j=1, \ldots,m)$. Then, we say that $\Gamma$ is in a {\it normal form} if for each $j = 1,\ldots,m$, there exists a $\tau$-component $G_j$ of $\Gamma^j$ such that every $\sigma$-edge
 of $\Gamma^j$ is either a free $\sigma$-edge or connected with $G_j$, and all of the half $\sigma$-edges connected with $G_j$ have orientations either toward or from $G_j$. We call $G_j$ the {\it initial $\tau$-component} of $\Gamma^j$. 
 
\end{definition}
See Figure \ref{20240330-1} for an example. 

\begin{figure}[ht]
\includegraphics*[height=5cm]{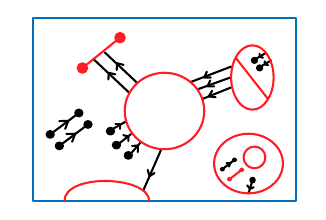}
\caption{An example of a $2$-chart in a normal form, where we omit labels, which are all $1$. The $\tau$-edge in the form of the circle in the center is the initial $\tau$-component associated with the region whose closure contains the larger part of $\partial B^2$.}
\label{20240330-1}
\end{figure}

\begin{proposition}\label{thm6-3}
Any 2-chart is C-move equivalent to a $2$-chart in a normal form given in Definition \ref{normal-form-chart}.
\end{proposition}

\begin{proof}
Let $\Gamma$ be a $2$-chart. The labels of edges are all one. We denote by the same notation $\Gamma$ the charts obtained from $\Gamma$ by C-moves.
 From now on, we consider each subchart $\Gamma^j=\Gamma \cap R_j$ of $\Gamma$ as given in Definition \ref{normal-form-chart}. We fix a $\tau$-component $G_j$ of $\Gamma^j$. 
 Let $\mathbf{e}$ be a $\sigma$-edge of $\Gamma^j$. Take a simple path $l$ from a point $x_0$ of $\mathrm{Int}(\mathbf{e})$ to an interior point $y_0$ of a $\tau$-edge of $G_j$ such that the intersection $\Gamma^j \cap \mathrm{Int} (l)$ consists of a finite number of transverse intersection points of interiors of $\sigma$-edges and $\mathrm{Int} (l)$, which will be denoted by $x_1, \ldots, x_s$ from $x_0$ to $y_0$. Consider the $\sigma$-edge $\mathbf{e}_s$ containing $x_s$. Applying a (D1)-move and an (A2)-move as in Figure \ref{20241222}, we deform $\mathbf{e}_s$ into a $\sigma$-edge or a pair of $\sigma$-edges connecting with $G_j$. 
 Repeating this process to the the $\sigma$-edges containing $x_k$ $(k=s, s-1, \ldots, 0)$ in this order, we deform $\mathbf{e}$ into a $\sigma$-edge or a pair of $\sigma$-edges connecting with $G_j$. Applying this deformation to $\sigma$-edges of $\Gamma^j$ other than free $\sigma$-edges, we deform $\Gamma^j$ to the form such that all of the $\sigma$-edges other than free $\sigma$-edges are connected with $G_j$. 
 Further, by (A2)-moves as in Figure \ref{20251005-2}, we deform pairs of half $\sigma$-edges with opposite orientations into free $\sigma$-edges; thus we deform $\Gamma^j$ so that all of the $\sigma$-edges other than free $\sigma$-edges are connected with $G_j$, and all of the half $\sigma$-edges are oriented toward/from $G_j$.  
 Thus we arrange $\Gamma$ into a normal form.  
\end{proof}

\begin{figure}[ht]
\includegraphics*[height=2.5cm]{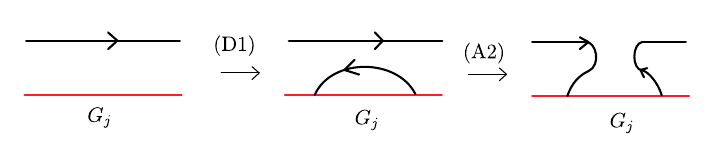}
\caption{Deforming a $\sigma$-edge in $\Gamma^j$ to a $\sigma$-edge connecting with the initial $\tau$-component $G_j$. We omit the labels of the $2$-charts. }
\label{20241222}
\end{figure}

\begin{figure}[ht]
\includegraphics*[height=2.4cm]{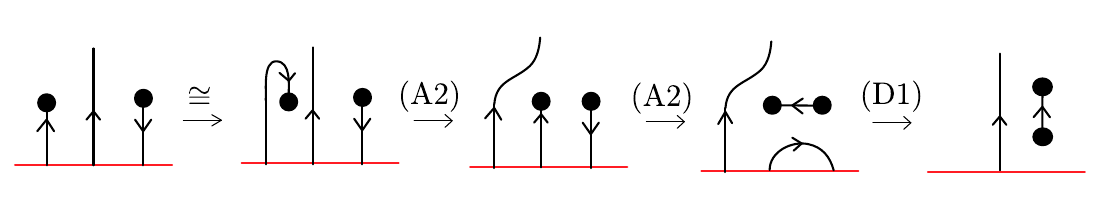}
\caption{Deforming a pair of half $\sigma$-edges with opposite orientations into a free $\sigma$-edge. We omit the labels of the $2$-charts. }
\label{20251005-2}
\end{figure}

\begin{remark}
We can deform a 2-chart $\Gamma$ to a more regulated form, also denoted by  $\Gamma$, satisfying that (1) every black  $\tau$-vertex is an endpoint of a free $\tau$-edge, (2) free $\sigma$-/$\tau$-edges are in a $2$-disk which contains no other parts of $\Gamma$, (3) all of the $\sigma$-edges connecting a $\tau$-component and the initial $\tau$-component $G_j$ have orientations either toward or from $G_j$, and (4) there are no $\sigma$-edges both of whose endpoints are connected with the initial $\tau$-component $G_j$. 
\end{remark}

\subsection{The plat closure of knitted surfaces}\label{sec4-3-pcl}
Let $S$ be a knitted surface of degree $2$. We define the {\it plat closure} of $S$ in such a way as follows. If $S$ is a braided
 surface, then $S$ is adequate; thus we define the plat closure  $\tilde{S}$ of $S$ by that of a braided surface, given by $\tilde{S}=S \cup A_\beta \subset \mathbb{R}^4$, where $A_\beta$ is the surface of (\ref{eq1218}) in Section \ref{sec:pcl-braided-surface}. If $S$ is a knitted surface of degree 2 such that $\partial S$ is the closure of a $2$-braid, then we define the plat
 closure of $S$ by the same way, given by $\tilde{S}=S \cup A_\beta$. We consider the case that $\partial S$ is the closure of a
 $2$-knit having hook pairs. Then, $\partial S$ is an $m$-component trivial link, where $m$ is the number of hook pairs in $\partial S$. We define the plat closure of $S$ by the closed surface obtained from $S$ by attaching $m$ copies of a $2$-disk trivially along $\partial S$.

 For knitted surfaces of degree $n$, we do not give the exact definition in terms of wickets in this paper, but we briefly give the idea of the plat closure, as follows. 
We call a $2n$-knit {\it adequate} if it is an element of the submonoid of the $2n$-knit monoid $D_{2n}$ generated by adequate $2n$-braids and $\tau_{2k-1}$ $(k=1, \ldots, n)$. 
We call a knitted surface $S$ of degree $2n$ {\it adequate} if $\partial S$ is the closure of an adequate $2n$-knit. 
For an adequate knitted surface $S$ of degree $2n$, we construct the {\it plat closure} of $S$ as follows. 
Let $\beta$ be the adequate $2n$-knit whose closure is $\partial S$.  
We denote by $\beta'$ the $2n$-braid obtained from the adequate $2n$-knit $\beta$ by changing each letter in the form of $\tau_i$ for some $i$ to $\sigma_i$. We denote by $s_1, \ldots,  s_m$ the letters of $\beta'$ coming from the letters of $\beta$ in the form of $\tau_{2k-1}$ $(k=1,\ldots, n)$. The $2n$-braid $\beta'$ is adequate. Let $A_{\beta'}$ be the surface associated with $\beta'$ given in (\ref{eq1218}) in Section \ref{sec:pcl-braided-surface}. From now on, we ignore the strings except for the $i$th and $(i+1)$th, when $s_j$ in $\beta'$ corresponds to $\tau_i$ in $\beta$. The part of $A_{\beta'}$ corresponding to $s_j$, that is $A_{\beta'} \cap (D^2 \times I \times [t_j, t_j']):=B_j$ for some $0\leq t_j<t_j'\leq 1$ with $\partial A_{\beta'} \cap (D^2 \times \{0\} \times [t_j, t_j'])=s_j$, is a band with one twist. 
Then, since $(A_{\beta'} \cap (D^2 \times I \times \partial [t_j, t_j'])) \cup (\partial S \cap (D^2 \times \{0\} \times [t_j, t_j'])):=L_j$ is a two-component trivial link, we take a disjoint union of two disks $D_j$ in $S^3=\partial (D^2 \times I \times [t_j, t_j'])$ whose boundary is $L_j$. Then, let $A'$ be the surface obtained from deleting the twisted band $B_j$ from $A_{\beta'}$ $(j=1, \ldots, m)$ and pasting the disks obtained from $D_j$ by pushing them into $D^2 \times I \times [t_j, t_j']$ $(j=1, \ldots, m)$. We see that the surface $A'$ is contained in $D^2 \times N$ such that $\partial A'=\partial S=S \cap A'$. 
We define the {\it plat closure} of $S$ by the surface-link $S \cup A'$ in $\mathbb{R}^4$. 

\begin{remark}\label{remark1223}
The plat closure of a knitted surface $S$ of degree $2$ is a special case of the closure of a $2$-dimensional knit of degree $2$ by the following meaning. Let $\Gamma$ be a $2$-chart of $S$, $y \in \partial B^2$ a point disjoint with $\Gamma$, and let $N$ be a small neighborhood of
 $y$ disjoint from $\Gamma$. Let $\tilde{B}^2$ be a $2$-disk such that $B^2$ is contained in the interior of $\tilde{B}^2$. Let ${\Gamma}'$ be the $2$-chart in $\tilde{B}^2$ obtained from $\Gamma$ by adding to $\Gamma$ a union of $\tau$-edges in the form $\partial B^2\backslash  (N\cup \Gamma)$.  Since ${\Gamma}'$ does not have boundary points,
 it represents a $2$-dimensional knit of degree $2$, denoted by $S'$. By applying the argument in the proof of Theorem \ref{Theorem: Alexander theorem for surface-links},
the plat closure of S and the closure of $S'$ is ambient isotopic as surface-links. In
 particular, when $\Gamma$ has no boundary points, then ${\Gamma}'$ is obtained from $\Gamma$ by an addition of a single free $\tau$-edge. 
\end{remark}

 We can calculate the Euler characteristic of a surface-link presented by the plat closure of a knitted surface of degree $2$ by using the information derived from 2-charts. The proof is left to the reader; see \cite[Exercise 9.15]{Kamada02}. 
For simplicity, we consider a surface-link presented by the closure of a 2-dimensional knit of degree 2, which is presented by 
a $2$-chart in a 2-sphere $S^2$.

\begin{proposition}
Let $F$ be a surface-link presented by the closure of a 2-dimensional knit of 
degree $2$, and let $\Gamma \subset S^2$ be its $2$-chart. Let $R_1, \ldots,R_m$ be  regions of the $\tau$-chart $\Gamma_\tau$, and put $\Gamma^j := \Gamma \cap R_j$, the subchart of $\Gamma$ $(j=1, \ldots,m)$. Then $F$ is a split union of $F_1, \ldots, F_m$, where $F_j$ is the plat closure of a knitted surface $S(\Gamma^j)$, $j=1, \ldots, m$, and the Euler characteristic $\chi(F_j)$ is calculated by 
\[
\chi(F_j)=2\chi(R_j')-b(\Gamma^j),
\]
where $b(\Gamma^j)$ is the number of black $\sigma$-vertices 
of $\Gamma^j$, and $\chi(R_j')$ is
the Euler characteristic of 
 $\mathrm{Cl}(R_j \cap (B^2 \backslash N(\Gamma_\tau)))$, where $N(\Gamma_\tau)$ is a regular neighborhood of $\Gamma_\tau$ $(j=1, \ldots, m)$. 
\end{proposition}

\subsection{Plat closure of knitted surfaces of degree $2$ presenting trivial surface-links}

A {\it standard 2-sphere} (respectively a {\it standard torus}) in $\mathbb{R}^4$ is the boundary of a 3-ball (respectively an unknotted solid torus) in $\mathbb{R}^3 \times \{0\}$, and a standard positive projective plane (respectively a {\it standard negative projective plane}) is the surface given by Figure \ref{20240330-2}(1) (respectively (2)). 
These surfaces are called {\it standard surfaces}. 
A surface-knot in $\mathbb{R}^4$ is called {\it trivial} if it is a connected sum of a finite number of standard surfaces in $\mathbb{R}^4$, and a surface-link is {\it trivial} if it is a split union of a finite number of trivial surface-knots.

\begin{figure}[ht]
\includegraphics*[height=6cm]{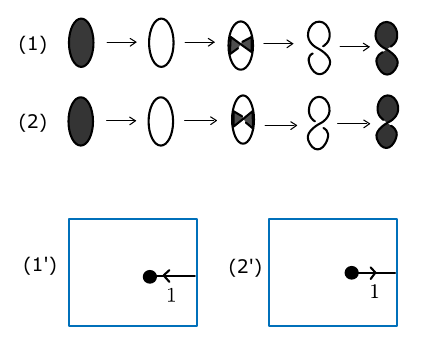}
\caption{(1) A standard positive projective plane $P_+$ and (2) a standard negative projective plane $P_-$.
Standard projective planes $P_+$ and $P_-$ are the plat closures of the knitted surfaces of degree 2 presented by 2-charts (1') and (2'), respectively.
 This figure is  \cite[Figure 41]{NY}. }
\label{20240330-2}
\end{figure}

When a half $\sigma$-edge has an orientation toward (respectively from) the black $\sigma$-vertex, we call it {\it positive} (respectively {\it negative}).
We denote by $T$ a standard torus, and by $P_+$ (respectively $P_-$) a standard positive (respectively negative) projective plane.

\begin{theorem}[\cite{NY}]\label{prop6-5}
Let $F$ be a trivial surface-knot, which is a connected sum of $g$ copies of $T$ and $m$ copies of $P_+$ and $n$ copies of $P_-$.
Then $F$ is the plat closure of the knitted surface presented by a $2$-chart $\Gamma$ such that $\Gamma$ consists of $\min\{m,n\}$ copies of a free $\sigma$-edge, $g$ copies of a free $\tau$-edge, 
and $|m-n|$ copies of a positive (respectively negative) half $\sigma$-edge if $m \geq n$ (respectively $m<n$).
\end{theorem}

In order to make this section of this paper self-contained, we show Theorem \ref{prop6-5} as follows; see \cite{NY}. 

\begin{proof}
First we observe that a standard 2-sphere is the plat closure of the 2-chart which is the empty graph. 
Next, by definition of the knitted surface $S(\Gamma)$ presented by a 2-chart $\Gamma$ and its plat closure, we see that a standard torus $T$ is the plat closure of the knitted surface presented by the 2-chart consisting of exactly one free $\tau$-edge. 
Let $\Gamma^+$ (respectively $\Gamma^-$) be the 2-chart consisting of exactly one positive (respectively negative) half $\sigma$-edge as given in Figure \ref{20240330-2}(1') (respectively (2')). 
Seeing Figure \ref{20240330-2}(1) and (2), we see that the plat closure of 
$S(\Gamma^{\pm})$ is $P_{\pm}$, respectively. 

For a pair of 2-charts in a disk, $\Gamma_1 \subset B_1^2$ and $\Gamma_2 \subset B_2^2$, let $J_1 \subset \partial B_1^2$ and $J_2 \subset \partial B_2^2$ be some pair of intervals which do not contain boundary points. We denote by $\Gamma_1 \cdot \Gamma_2$ the 2-chart in $B^2=B_1^2 \cup_{J_1=J_2} B_2^2$ obtained from $\Gamma_1$ and $\Gamma_2$ by identifying $J_1$ and $J_2$. Then, by construction, the plat closure of  $S(\Gamma_1\cdot \Gamma_2)$ is a connected sum of the plat closures of $S(\Gamma_1)$ and $S(\Gamma_2)$. 
Thus, the trivial surface-knot $S$ is presented by a $2$-chart $\Gamma$ consisting of $g$ copies of a free $\tau$-edge 
and $m$ (respectively $n$) copies of a positive (respectively negative) half $\sigma$-edge. When we have two opposite half $\sigma$-edges, then, applying (A2) and (A8)-moves, we change them into one free $\sigma$-edge. Applying this process for every pair of opposite half $\sigma$-edges, we have a 2-chart in the required form.
\end{proof}

 \begin{corollary}\label{cor5-5}
Any trivial surface-link is ambient isotopic to the plat closure of some knitted surface of degree $2$.
\end{corollary}
 
 \begin{proof}
 Let $F$ be a trivial surface-link, and let $F_1, \ldots, F_m$  denote connected components of $F$. Let $l_j$ be the segment 
 $I \times \{j/m\}$ in $B^2 = I \times  I$ $ (j= 1,\ldots,m-1)$. Then, $l_1\ldots,l_{m-1}$ divide $B^2$ into $m$ $2$-disks 
 $E_j=I \times [(j-1)/m, j/m]$,  
  $j = 1,\ldots,m$. Applying Theorem \ref{prop6-5}, we obtain a $2$-chart $\Gamma_j$ in $E_j$ such that the plat closure of
 a knitted surface $S(\Gamma_j)$ is ambient isotopic to $F_j$. Let $\Gamma$ be the union of $\Gamma_j$ $(j=1, \ldots, m)$ and $\tau$-edges in the form of $\cup_{j=1}^{m-1} l_j$. Then, the plat closure of $S(\Gamma)$ is a split union of the plat closures of $S(\Gamma_j)$, $j=1, \ldots, m$, which is ambient isotopic to $F$. 
\end{proof}

\subsection{Knit plat index and the knit index}
The {\it plat index} \cite{Yasuda21} of a surface-link $F$, denoted by $\mathrm{Plat}(F)$, is the half of the minimum of the degrees of braided surfaces whose plat closures are ambient isotopic to $F$. We define the {\it knit plat index}, denoted by $\mathrm{knit.Plat}(F)$, as the half of the minimum of the degrees of knitted surfaces whose plat closure is ambient isotopic to $F$.
By definition, for a surface-link $F$,
\[
\mathrm{knit.Plat}(F)\leq \mathrm{Plat}(F).
\]

It is known \cite{Yasuda21} that the plat index of a surface-link is one if and only if it is a trivial 2-sphere or a nonorientable trivial surface-knot, and the plat index is two if the surface-link is a trivial orientable surface-knot with positive genus. %
Corollary \ref{cor5-5} implies the following:

\begin{corollary}\label{cor5-7}
Let $F$ be any trivial surface-link. Then, the knit plat index of $F$ is one:
\[
\mathrm{knit.Plat}(F)=1.
\]
\end{corollary}

The {\it braid index} \cite{Kamada02} of an oriented surface-link $F$, denoted by $\mathrm{Braid}(F)$, is the minimum of the degrees of 2-dimensional braids whose closures are ambient isotopic to $F$. We define the {\it knit index} of a surface-link $F$, denoted by $\mathrm{Knit}(F)$, by the minimum of the degrees of 2-dimensional knits whose closures are ambient isotopic to $F$. By definition, for an oriented surface-link $F$,
\[\mathrm{Knit}(F) \leq \mathrm{Braid}(F). 
\]
By the argument in the proof of Theorem \ref{Theorem: Alexander theorem for surface-links}, we have the following. 

\begin{corollary}\label{cor:knit-index}
Let $F$ be any surface-link. Then, the knit index of $F$ is bounded from above by twice the plat index of $F$: 
\[
\mathrm{Knit}(F)\leq 2\mathrm{Plat}(F). 
\]
\end{corollary}

By Corollary \ref{thm20241217}, we have the following.
\begin{corollary}
Let $F$ be a trivial surface-link. 
Then, if $F$ is not a trivial 2-sphere, then, 
\[
\mathrm{Knit}(F)=2.
\]
\end{corollary}

\subsection{Remark on 2-charts presenting trivial surface-knots} The trivial surface-knot $P_+ \# P_-$ is presented by a 2-chart other than that given in Theorem \ref{prop6-5}. 

\begin{proposition}\label{prop6-7}
The trivial surface-knot $P_+ \# P_-$ is ambient isotopic to the plat closure of the knitted surfaces of degree $2$ presented by the 2-charts as in Figure \ref{20241212-2}. 
\end{proposition}

\begin{figure}[ht]
\includegraphics*[height=3cm]{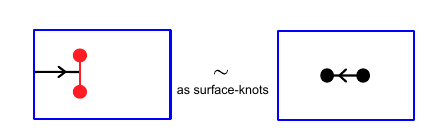}
\caption{A pair of $2$-charts, the plat closures of whose presenting knitted surfaces of degree 2 are ambient isotopic as surface-knots to $P_+\# P_-$, where we omit labels. The right chart is the chart given in Theorem \ref{prop6-5}.}
\label{20241212-2}
\end{figure}

\begin{proof}
The trivial surface-knot $P_+\# P_-$ has the motion picture as in Figure \ref{20241212-3}(1). The surface is ambient isotopic to the surface as in Figure \ref{20241212-3}(2). This surface is ambient isotopic to the surface as in Figure \ref{20241212-3}(3), which is the surface given by the gray regions of the 2-charts in Figure \ref{20241212-1}; so it is the plat closure of the knitted surface of degree 2 presented by the $2$-chart in the left side  of Figure \ref{20241212-2}. 
By Theorem \ref{prop6-5}, $P_+\# P_-$ is ambient isotopic to the plat closure of the knitted surface presented by the $2$-chart in the right side of Figure \ref{20241212-2}. 
Thus the $2$-charts as in Figure \ref{20241212-2} both present $P_+\# P_-$. 
\end{proof}

\begin{figure}[ht]
\includegraphics*[height=5cm]{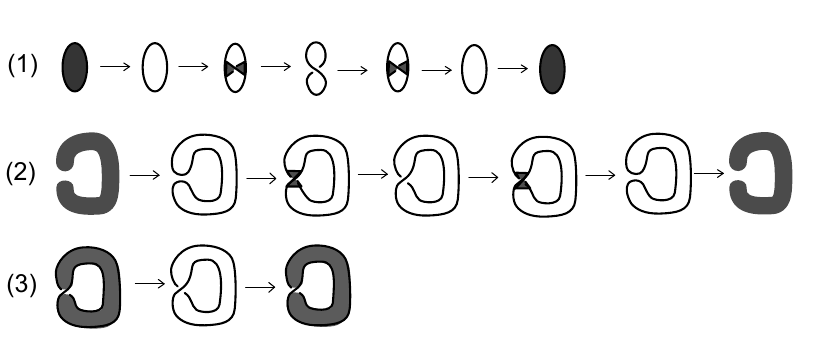}
\caption{Motion pictures of $P_+ \# P_-$.}
\label{20241212-3}
\end{figure}

\begin{figure}[ht]
\includegraphics*[height=3cm]{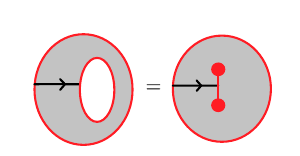}
\caption{The gray regions of the $2$-charts present the same surface of Figure \ref{20241212-3}(3). }
\label{20241212-1}
\end{figure}

Let $F$ be a trivial surface-knot, which is a connected sum of $g$ copies of $T$ and $m$ copies of $P_+$ and $n$ copies of $P_-$. 
If $m \neq n$, then $F$ is a connected sum of $\min\{m,n\}+g$ copies of $T$ and $|m-n|$ copies of $P_+$ (respectively $P_-$) if $m>n$ (respectively $m<n$). 
We can show this by using $2$-charts, applying Theorem \ref{prop6-5}, Proposition \ref{prop6-7} and C-moves. The proof is left to the reader as exercise. In the following, we give the result in terms of $2$-charts. 

\begin{corollary}
Let $\Gamma$ be a $2$-chart which is a disjoint union of  $\mu$  copies of a free $\sigma$-edge, $\nu$  copies of a free $\tau$-edge and   $\lambda$ copies of a positive (respectively negative) half $\sigma$-edge. If $\lambda \neq 0$, then the plat closure of the knitted surface of degree 2 presented by $\Gamma$ is ambient isotopic to the plat closure of the knitted surface of degree 2 presented by the $2$-chart $\Gamma'$ which is a disjoint union of  $\mu+\nu$ copies of a free $\tau$-edge and  $\lambda$ copies of a positive (respectively negative) half $\sigma$-edge. 
\end{corollary}

\subsection{Proof of Theorem \ref{thm20241213} and Corollary \ref{thm20241217}}

Let $F$ be a surface-link in $\mathbb{R}^4$. 
A {\it 1-handle attached to $F$} is a 3-ball $b=D^2 \times I$ embedded in $\mathbb{R}^4$ such that $b \cap F=(D^2 \times \partial I) \cap F$. The arc $\{0\} \times I \subset D^2 \times I$ is called the {\it core} of the 1-handle. The {\it 1-handle surgery} is the operation which changes $F$ to a new surface-link $F'$ given by 
$F'=\mathrm{Cl}(F \backslash (D^2 \times \partial I) \cup (\partial D^2 \times I))$. 
A 1-handle $b$ attached to $F$ is called {\it trivial} if we can move the 1-handle $b$ by ambient isotopy so that $b$ is contained in a 4-ball $U$ in $\mathbb{R}^4$ such that the pair $(U, F \cap U)$ is homeomorphic to the standard $(4,2)$-ball pair $(B^4, B^2)$. 
   
\begin{theorem}\label{thm6-7}
Let $F$ be the plat closure of an arbitrary knitted surface of degree 2. Then, $F$ is a trivial surface-link.
\end{theorem}

\begin{proof}
Let $\Gamma$ be an arbitrary $2$-chart. 
It suffices to show the result for the subchart of each region divided by the $\tau$-chart $\Gamma_\tau$ of $\Gamma$. 
So we assume that $\Gamma \subset B^2$ is 
in a normal form such that $B^2\backslash \Gamma_\tau$ is connected. 
We call a $\sigma$-edge each of whose endpoints is either a boundary point or a mixed trivalent vertex a {\it conjunctive $\sigma$-edge}. 
If $\Gamma$ has no conjunctive $\sigma$-edge, Theorem \ref{prop6-5} implies that $F$ is a trivial surface-knot. 
Let $l_1, \ldots, l_s$ ($s>0$) be segments in $B^2$ such that each $l_j$ coincides with each conjunctive $\sigma$-edge. 
Since we take the plat closure, we regard $\partial B^2$ as the initial $\tau$-component.

Let $\Gamma'$ be the 2-chart obtained from $\Gamma$ by changing each conjunctive $\sigma$-edge into a $\tau$-edge. Let $F$ (respectively $F'$)  be the plat closure of the knitted surface of degree 2 presented by $\Gamma$ (respectively $\Gamma'$).
We see that $F'$ is a split union of the plat closures of  knitted surfaces of degree 2 presented by $2$-charts of regions divided by the $\tau$-chart $\Gamma'_{\tau}$ of $\Gamma'$, and $F$ is obtained from $F'$ by 1-handle surgery along 1-handles $b_1, \ldots, b_s$ such that the projection to $B^2$ of the core of $b_j$ transversely intersects the segment $l_j$ $(j=1, \ldots, s)$. We remark that since $\Gamma$ is in a normal form, each region divided by $\Gamma'_{\tau}$ is a 2-disk. We consider the case $s=1$, when the number of conjunctive $\sigma$-edges is one. We put $b:=b_1$, the 1-handle along which $F$ is obtained from $F'$ by 1-handle surgery. 
The plat closure $F$ is a surface-knot, and $F'$ is either a surface-knot or a surface-link with two components. 
By Theorem \ref{prop6-5}, using (T2)-moves (see Figure \ref{20241211-1a}) if necessary, we see that $F'$ is a trivial surface-knot/link. If $F'$ is a surface-link with two components, denoted by $F_1'$ and $F_2'$,  then, $F$ is a connected sum of $F_1'$ and $F_2'$. If $F'$ is a surface-knot, then, since $F'$ is trivial, the 1-handle $b$ is trivial  \cite{Kamada2014}. 
Hence, the result of 1-handle surgery, $F$, is a connected sum of $F'$ and either a standard torus $T$ or a connected sum $P_+ \# P_-$ (see Figures \ref{20241212-3}(3) and \ref{20241212-1}). For the general case, since $F'$ is a trivial surface-link by Theorem \ref{prop6-5}, by induction on the number $s$ of conjunctive $\sigma$-edges, we see that $F$ is obtained from components of $F'$ by taking connected sum with several copies of $T$ and $P_+ \# P_-$. Hence we see that $F$ is a trivial surface-knot. Thus we have the required result.
\end{proof}

\begin{proof}[Proof of Theorem \ref{thm20241213}]
By Corollary \ref{cor5-5} and Theorem~\ref{thm6-7}, we have the required result. 
\end{proof}

\begin{proof}[Proof of Corollary \ref{thm20241217}]
 By Remark \ref{remark1223}, 
Corollary \ref{cor5-5} implies that any trivial surface-link is ambient isotopic to the closure of some 2-dimensional knit of degree 2. 
Let $F$ be a 2-dimensional knit of degree 2, and 
let $\Gamma$ be a 2-chart in $B^2$ which presents $F$, that is, $\Gamma$ satisfies $\Gamma \cap \partial B^2=\emptyset$. 
If $\Gamma$ contains at least one $\tau$-edge $\mathbf{e}$, we denote by $\Gamma_\tau^*$ the $\tau$-component of the $\tau$-chart $\Gamma_\tau$ containing $\mathbf{e}$. We see that $(B^2/\partial B^2)\backslash \Gamma_\tau^*=S^2\backslash \Gamma_\tau^*$ consists of connected components $D_1, \ldots, D_s$ for some $s>0$, such that $D_j$ $(j=1,\ldots, s)$ is regarded as an open disk. Then, the closure of $F$ is the split union of the plat closures of knitted surfaces of degree 2 presented by $\Gamma \cap \mathrm{Cl}(D_j) \subset \mathrm{Cl}(D_j)$ $(j=1, \ldots, s)$. We remark that $\Gamma \cap \mathrm{Cl}(D_j)$ may not present a 2-dimensional knit of degree 2. Theorem \ref{thm20241213} implies that the plat closures are trivial surface-links; thus $F$ is a trivial surface-link. 
If $\Gamma$ does not contain $\tau$-edges, then it presents a 2-dimensional braid, the closure of which is an orientable trivial surface-link \cite{Kamada02}. 
\end{proof}

\section*{Acknowledgements}
Gratitude is expressed to the referee for kindly reviewing this paper. The first author would like to thank Professor Masahiko Saito for suggesting the topic in Section 5.5, and Professor Takashi Ichikawa for his helpful advice. 
The first author was partially supported by JST FOREST Program, Grant Number JPMJFR202U. The second author was partially supported by JSPS KAKENHI Grant Number 22J20494.

\end{document}